\documentclass{article}
\usepackage[utf8]{inputenc}
\usepackage{fullpage}

\usepackage{graphicx}
\usepackage{hyperref}
\usepackage{amsmath,amssymb,amsthm,amsfonts}
\usepackage{todonotes}
\usepackage{comment}
\usepackage{breqn}
\usepackage{csquotes}
\usetikzlibrary{automata, positioning,calc,decorations.pathreplacing,angles,quotes}

\makeatletter
\def\blfootnote{\xdef\@thefnmark{}\@footnotetext}
\makeatother

\newtheorem{thm}{Theorem}[section]
\newtheorem{cor}{Corollary}[section]

\newtheorem{prop}{Proposition}[section]
\newtheorem{lem}{Lemma}[section]

\numberwithin{equation}{section}

\theoremstyle{remark}
\newtheorem{rem}{Remark}

\DeclareMathOperator{\nul}{Null}
\newcommand{\Z}{{\mathbb Z}}
\newcommand{\R}{{\mathbb R}}
\newcommand{\N}{{\mathbb N}}

\newcommand{\email}[1]{\href{mailto:#1}{\texttt{{\small #1}}}}

\newcounter{row}
\newcounter{col}

\newcommand\setrow[9]{
  \setcounter{col}{1}
  \foreach \n in {#1, #2, #3, #4, #5, #6, #7, #8, #9} {
    \edef\x{\value{col} - 0.5}
    \edef\y{9.5 - \value{row}}
    \node[anchor=center] at (\x, \y) {$\n$};
    \stepcounter{col}
  }}

\title{Finite-Memory Elephant Random Walk and the Central Limit Theorem for Additive Functionals\footnote{Research performed during \href{http://markov-chains-reu.math.uconn.edu}{Markov Chains REU}, partially supported by NSA grant H98230-19-1-0022 to Iddo Ben-Ari.}} 
\author{Iddo Ben-Ari \\ \email{iddo.ben-ari@uconn.edu}
\and 
Jonah Green\\ \email{jonahgreen1129@yahoo.com} \and Taylor Meredith \\ \email{taylor.meredith@nyu.edu} \and   Hugo Panzo\footnote{Supported at the Technion by a Zuckerman Fellowship.} \\ \email{panzo@campus.technion.ac.il}\and Xiaoran  Tan\\ \email{xiaoran.tan@uconn.edu}}
\date{\today}
\begin{document}
\maketitle

\abstract{The Central Limit Theorem (CLT) for additive functionals of Markov chains is a well known result with a long history.  In this paper we present applications to two finite-memory versions of the Elephant Random Walk, solving a problem from \cite{Gut18}. We also present a derivation of the CLT for additive functionals of finite state Markov chains, which is based on positive recurrence, the  CLT for IID sequences and some elementary linear algebra, and which focuses on characterization of the variance.}  
\section{Introduction} 
\subsection{The Central Limit Theorem for Additive Functionals} 
Let  $P$ be  a transition function on a finite state space. For simplicity, we will assume that the state space is  $\{1,\dots,N\}$ where $N\in\N$.  We will treat $P$ as an $N\times N$ stochastic matrix (sum of each row is $1$) with $P(i,j)$ being the entry at the $i$-th row and $j$-th column, representing transition from state $i$ to state $j$.  We write $P^n$ for the $n$-th power of the matrix $P$,  $P^{n+1}=  P^n P $.  We will assume that $P$ is {\it irreducible}, that is, for  every $i,j$, there exists $n=n(i,j)$ such that
\begin{equation} 
\label{eq:irreducible} 
P^n (i,j)>0.
\end{equation} 
Under \eqref{eq:irreducible}, $P$ has a unique stationary distribution $\pi$ (see Section \ref{seC:potential}). In the sequel, if  $u:\{1,\dots,N\}\to \R$, we write  $\pi(u)$ meaning $\sum_i \pi(i) u(i)$. 

Let ${\bf X}=(X_0,X_1,\dots)$ be a Markov chain with transition function $P$. That is, $X_0,X_1,\dots$ are random variables with the property
 \begin{equation} 
 \label{eq:markov_property}
 \mathbb{P}(X_{n+1}=j | X_{n}=i_n,\dots,X_{0}=i_0) = P(i_n,j).
 \end{equation} 
 Equation \eqref{eq:markov_property} is known as the {\it Markov property}. We adopt the usual notation $\mathbb{P}_i(\cdot):=\mathbb{P}(\cdot |X_0=i)$ and use $\mathbb{E}_i[\cdot]:=\mathbb{E}[\cdot |X_0=i]$ for the corresponding expectation. More generally, we use $\mathbb{P}_\mu(\cdot)$ and $\mathbb{E}_\mu[\cdot]$ when the starting point has distribution $\mu$.

 For a function $f:\{1,\dots,N\}\to \R$, define the {\it additive functional}
 \begin{equation} 
 \label{eq:additive_functional} 
 I_n (f)= \sum_{k=0}^{n-1} f(X_k).
 \end{equation}
  Note that for any given $f$, 
$$
\bar f := f-\pi(f){\bf 1}
$$
 satisfies $\pi(\bar f)=0$.  
 Suppose $ \pi(g)=0$. Then 
 \begin{equation}
 \label{eq:unique_sol}
 \begin{cases} (I-P) u = g\\ \pi( u)=0\end{cases}
 \end{equation} 
 has a unique solution we denote by $\bar U_g$ (Corollary \ref{cor:unique_sol}).
 \begin{lem}
 \label{lem:var_lemma} 
 Suppose $f:\{1,\dots,N\}\to \R$ and $\mu$ is any probability measure on $\{1,\dots,N\}$ and let $\bar f$ and $\bar U_{\bar f}$ be as above. Define 
 $$ \sigma^2_{f}= \pi\left( (\bar U_{\bar f})^2 - (P\bar U_{\bar f} )^2\right).$$
 \begin{enumerate}
     \item If $(I-P)u=\bar f$, then
     \begin{equation} 
     \label{eq:general_variance} \sigma^2_f = \pi\left( u^2 - (Pu)^2\right).
     \end{equation}
     \item $\displaystyle\sigma^2_f = \lim_{n\to\infty} \frac{\mathbb{E}_{\mu} [I_n (\bar f)^2]}{n}$
 \end{enumerate}
 \end{lem} 
  The following theorem is well-known.  Our goal in presenting a proof of the results in their  weakest form  is to provide a reference that is intuitive and accessible to non-experts, building on the all-familiar Central Limit Theorem (CLT) for IID random variables and focusing on the derivation of \eqref{eq:general_variance} through elementary linear algebra. We believe that the very broad range of applications of finite-state Markov chains merits such a presentation, though this specific work was motivated by our study of finite-memory versions of the Elephant Random Walk which are discussed in Section \ref{sec:elephant}. 
 \begin{thm}
 \label{th:CLT}
 Suppose $P$ is irreducible on a finite state space and let $f:\{1,\dots,N\}\to \R$ with $\bar f$ and $\sigma_f^2$ as above.  Then   \begin{enumerate} 
 \item (Law of Large Numbers) $\displaystyle \frac{I_n (\bar f)}{n}\to 0$, a.s. 
 \item (Central Limit Theorem) $\displaystyle \frac{I_n (\bar f)}{\sqrt{n}} \Rightarrow N(0,\sigma_f^2)$.
 \end{enumerate}
 \end{thm}
The theorem has a number of proofs under weaker assumptions, including more detailed results (e.g. functional CLT).  The most effective treatment of the CLT to date is through the martingale CLT,  \cite{varadhan_book}. Specific applications to additive functionals of Markov chains can be found in \cite{kipnis},\cite{maxwell},\cite{CLT_peligrad} as well as in many other resources, such as the review paper \cite{sep} (in continuous time).  The proof we present in this paper is different and more rudimentary,  based on (positive) recurrence. We believe it also explains the expression for the variance in an intuitive way, through linear algebra and moment calculations.  Using recurrence to prove the CLT for additive functionals can be  also found in  \cite[p. 416]{meyn}, where the result is more general than ours, yet  does not provide the expression for the variance \eqref{eq:general_variance}.  

\subsection{Proof of  Theorem \ref{th:CLT} and Lemma \ref{lem:var_lemma}}
\subsubsection{Outline} 
The proof has three parts. In the first  part, Section \ref{seC:potential}, we introduce the potential function and use some linear algebra to study its properties. This discussion will then be used in the  characterization of the variance in Theorem \ref{th:CLT}. In the second part, Section \ref{sec:IID}, we use the recurrence structure of the Markov chain and the classical Law of Large Numbers (LLN) and CLT for IID sequences to prove the theorem along  (random) subsequences. In the third part, Section \ref{sec:subseq},  we extend from subsequences to all sequences, completing the proof of the theorem. The first statement in Lemma \ref{lem:var_lemma} is Corollary \ref{cor:variance}, and we outline the proof to the second statement in Section \ref{sec:subseq}. We only outline this because the result is well-known and because the verification is routine and is of a similar nature to the calculation we use for the CLT. 
\subsubsection{The potential function} 
\label{seC:potential}
 Functions from the state space $\{1,\dots,N\}$ will be considered as column vectors. That is, if $f:\{1,\dots,N\}\to \R$, then $f$ is identified with the $N\times 1$ vector $(f(1),\dots,f(N))^t$, where $t$ represents the transpose. We write ${\bf 1}$ for the constant function ${\bf 1}(i)=1$ for all $i$.  Measures on the state space will be considered as row vectors. 

Since $P$ is stochastic, $P {\bf 1}={\bf 1}$. That is, ${\bf 1}$ is an eigenvector for $P$ corresponding to the eigenvalue $1$. Let $v$ be a nonzero vector satisfying $P v =v$. Hence $P^n v=v$ for all $n\in\mathbb{N}$ as well. We will show that $v$ is constant. Let $i$ be such that $v(i)\ge v(i')$ for all $i'$. Without loss of generality we may assume $v(i)\ge 0$. We can write  
$$v(i) = \sum_{j}P^n(i,j) v(j) =v(i) + \sum_{j} P^n (i,j) \Big(v(i) - v(j)\Big)$$
which holds for all $n\in\mathbb{N}$. Since the sum on the right-hand side must be zero with each summand nonnegative, it follows that $v(j)=v(i)$ if $P^n(i,j)>0$. Irreducibility \eqref{eq:irreducible} now implies that $v(j)=v(i)$ for all $j$. 

From this it follows that $\nul (I-P)$, the null space of $I-P$, is spanned by ${\bf 1}$. Therefore, the image of $I-P$ has dimension $N-1$, and its orthogonal complement is one-dimensional. A row vector $v$ is in this  orthogonal complement if and only if it is orthogonal to (the transpose of) each of the columns of  $I-P$, which is equivalent to  $v(I- P)=0$, or simply  $vP =v$. Fix nonzero $v$ in this subspace. By the triangle inequality, $|v(j)|\le \sum_{i} |v(i)| P(i,j)$. Summing over $j$, we have $\sum_{j} |v(j)|\le\sum_{i} |v(i)|$. Since this is actually an equality, it follows that $|v(j)|=\sum_{i} |v(i)| P(i,j)$ for all $j$. Therefore $v$ does not change sign.  As a result of the irreducibility, all entries of $v$ are nonzero. We let $\pi$ be the element in this orthogonal complement normalized to be a probability measure. We call $\pi$ the {\it stationary distribution} for $P$. The condition $\pi (I-P)=0$ is equivalent to $\pi P = P$. Of course, this implies $\pi P^n = P$ for all $n$. In terms of the Markov chain itself, if $X_0$ has distribution $\pi$, then $X_1$ has distribution $\pi$, and so do $X_2,X_3,\dots$. This is why $\pi$ is called the stationary distribution. We also observe that 
\begin{prop}\label{cor:unique_sol}
\leavevmode
\begin{enumerate} 
\item The equation $(I-P)u=g$ has a solution if and only if $\pi(g)=0$, and any two solutions differ by a constant.
\item In particular, if $\pi(g)=0$, equation \eqref{eq:unique_sol} has a unique solution. 
\end{enumerate}
\end{prop}
\begin{proof}
The first statement is clear because the image of $(I-P)$ is the orthogonal complement of the span of $\pi$ and the null space of $(I-P)$ consists of constants. 

For the second statement, let $u$ and $v$ be any solutions to equation \eqref{eq:unique_sol}. Then by the first statement we have $u-v=c{\bf 1}$. Applying $\pi$ to both sides of this results in $\pi(u-v)=c$ which implies $c=0$ by hypothesis. It follows that $u=v$.  
\end{proof} 

Next we add a little more of probability into the mix. In what follows, we first fix a state $i_0$ and then define the \emph{hitting time} of $i_0$ by 
\begin{equation} 
\label{eq:Ti0}T_{i_0} = \inf\{n\ge 1: X_n = i_0\}.
\end{equation} 
In words, $T_{i_0}$ is the first time the process visits $i_0$ after time $0$. By irreducibility and since the state-space is finite, there exists $k\in\N$ and $0<q<1$ such that $\mathbb{P}_{i}(T_{i_0}>k)\leq q$ for all $i$. From this it follows that 
\begin{equation}
\label{eq:finite_time}
\max_i \mathbb{P}_i( T_{i_0}>n k) \le q^n,
\end{equation} 
and in particular, $T_{i_0}$ has finite expectation (and finite MGF in some open interval containing $0$). 

Let $f:\{1,\dots,N\}\to {\mathbb R}$ and define the function of $i$
\begin{equation}
\label{U_f}
U_f(i,i_0)= \mathbb{E}_i\left[I_{T_{i_0}}(f)\right] =\mathbb{E}_i \left[ \sum_{k=0}^{T_{i_0}-1} f(X_k)\right].
\end{equation}
In words, $U_f(i, i_0)$ is the average of the sum of values of $f$ along the path of ${\bf X}$ started at $i$ up to one step before it hits $i_0$ for the first time after time $0$. We refer to $U_f$ as the \emph{potential} of $f$.

We wish to get rid of the random limit in the summation, and change the order of summation and integration. This gives: 
\begin{equation}
\label{U_f with indicator}
\begin{aligned}  
U_f(i,i_0) 
&= \mathbb{E}_i \left[ \sum_{k=0}^\infty {\bf 1}_{\{T_{i_0}>k\}}f(X_k) \right]\\ 
&= \sum_{k=0}^\infty \mathbb{E}_i\left[f(X_k),T_{i_0}>k\right].
\end{aligned}
\end{equation}

\begin{lem}\label{lem:U_map} 
The mapping $f\to U_f(\cdot,i_0)$ is linear and satisfies the following: 
\begin{enumerate} 
\item $\displaystyle U_{(I-P)f}(\cdot,i_0)= f (\cdot)-f(i_0)\bf{1}$\\
\item $\displaystyle U_{\delta_{i_0}}(\cdot,i_0) =\delta_{i_0}(\cdot)$
\end{enumerate}
In particular, if $h$ is in the image of $I-P$, then $U_h (i_0,i_0)=0$ and $(I-P)U_h =h$.
\end{lem} 
\begin{proof}
The linearity follows from the definition and so does the second identity. To prove the first identity, we recall that $Pf (i) = \sum_{j} P(i,j) f(j)$ and use \eqref{U_f with indicator} to write 
$$ U_{Pf} (i,i_0) = \sum_{k=0}^\infty \mathbb{E}_i\left[ \sum_{j} P(X_k,j)f(j),T_{i_0}>k\right].$$
From the Markov property \eqref{eq:markov_property}, 
$$ \mathbb{E}_i\left [\sum_{j\ne i_0} P(X_k,j) f(j) ,T_{i_0}>k\right] = \mathbb{E}_i\left[f(X_{k+1}),T_{i_0}>k+1\right]$$ 
and 
$$ \mathbb{E}_i\left[ P(X_k,i_0)f(i_0),T_{i_0}>k\right]=f(i_0) \mathbb{P}_i(T_{i_0}=k+1).$$ 
Therefore, 
\begin{align*}
U_{Pf} (i,i_0)&=f(i_0)\sum_{k=0}^\infty \mathbb{P}_i(T_{i_0}=k+1)+\sum_{k=0}^\infty \mathbb{E}_i\left[f(X_{k+1}),T_{i_0}>k+1\right]\\
&=f(i_0)-\mathbb{E}_i\left[f(X_0),T_{i_0}>0\right]+\sum_{k=0}^\infty \mathbb{E}_i\left[f(X_{k}),T_{i_0}>k\right]\\
&=f(i_0)-f(i)+U_f (i,i_0),
\end{align*}
which by linearity gives the first equation. 
\end{proof}

Notice that by Lemma \ref{lem:U_map} we have
\[
U_{(I-P) ({\bf 1}-\delta_{i_0}) + \delta_{i_0}} = {\bf 1}
\]
and this equation can be simplified to
\[
U_{P\delta_{i_0}} = {\bf 1}.
\]
Hence for any $f$ we have
\begin{equation}\label{eq:Uf_all}
U_{(I-P)f + P\delta_{i_0} f} =f.
\end{equation} 
This shows that $f\to U_f (\cdot,i_0)$ is the left (hence also right) inverse of the linear mapping $(I-P) +P\delta_{i_0}$. 
\begin{cor}
\label{cor:uh_inverse}
For any $f$ we have
\begin{equation}\label{eq:inverse}
 f(\cdot) = \Big((I-P) + P\delta_{i_0}\Big) U_f(\cdot,i_0).
 \end{equation}
\end{cor}

Since  $\pi$ is orthogonal to the image of $I-P$, applying $\pi$ to both sides of \eqref{eq:inverse} results in  
\begin{align*}
\pi (f) &=\pi\Big(P\delta_{i_0} U_f(\cdot,i_0)\Big)\\
&= \pi(i_0) U_f (i_0,i_0).
\end{align*}
Setting $f={\bf 1}$, we get 
$$ 1= \pi({\bf 1}) = \pi(i_0) U_{\bf 1}(i_0,i_0)$$
which leads to the following representation for $\pi$: 
\begin{cor}\label{cor:stat}
\leavevmode
\begin{enumerate}
    \item $\displaystyle \pi(i_0) = \frac{1}{U_{\bf 1}(i_0,i_0)} = \frac{1}{\mathbb{E}_{i_0}[T_{i_0}]}$. 
    \item $\displaystyle \pi(f) = \frac{U_f (i_0,i_0)}{U_{\bf 1}(i_0,i_0)}= \mathbb{E}_{i_0} \left[\displaystyle \sum_{k=0}^{T_{i_0}-1} f(X_k)\right]\Bigg/\mathbb{E}_{i_0}[T_{i_0}]$. 
\end{enumerate}
\end{cor}
Note that the choice of $i_0$ is arbitrary, and therefore for any $i$, 
$$ \pi(i)= \frac{1}{\mathbb{E}_{i} [T_i]}.$$ 

We now  study the expectation of the square of $I_{T_{i_0}}(f)$. This is the heart of the variance from Theorem \ref{th:CLT}. For any $i$, define 
\begin{equation}
U_f^2 (i,i_0) = \mathbb{E}_i\left[\left(I_{T_{i_0}}(f)\right)^2\right] = \mathbb{E}_i \left [ \left(\sum_{k=0}^{T_{i_0}-1} f(X_k)\right)^2\right].
\end{equation}
For a function $Z$ of the process ${\bf X}$ we define 
\begin{equation*}
    \sigma_i (Z) = \mathbb{E}_i \left[\left(Z - \mathbb{E}_{i}[Z]\right)^2\right],
\end{equation*}
the variance of $Z$ under $P_i$. In particular, 
\begin{equation}\label{eq:sigma_I}
    \sigma^2_{i_0}\left(I_{T_{i_0}}(f)\right) = U_f^2(i_0,i_0) - \Big(U_f(i_0,i_0)\Big)^2.
\end{equation}

The following lemma expresses $U_f^2$ as the potential of a different function.
\begin{lem}\label{lem:U2f}
For any $i$ we have
\begin{align*} 
  U_{f}^2 (i,i_0) & = U_{2f U_f (\cdot,i_0)-f^2}(i,i_0).
\end{align*}
\end{lem}
\begin{proof}
Let's start by rewriting: 
\begin{align} 
\nonumber 
U_{f}^2(i, i_0)
&= \mathbb{E}_{i}\left[\left(\sum_{k=0}^{T_{i_0}-1}f(X_k)\right)^2\right] \\
\nonumber
&= \mathbb{E}_{i}\left[\left(\sum_{k=0}^{\infty}{\bf 1}_{\{T_{i_0}>k\}}f(X_k)\right)^2\right] \\
\nonumber
&= \mathbb{E}_{i}\left[\sum_{k=0}^{\infty}\sum_{k'=0}^{\infty}{\bf 1}_{\{T_{i_0}>k\}}{\bf 1}_{\{T_{i_0}>k'\}}f(X_k)f(X_{k'})\right] \\
\label{eq:start_square}
& = 2 \mathbb{E}_{i}\left[\sum_{k=0}^{\infty}\sum_{k'=k}^{\infty}{\bf 1}_{\{T_{i_0}>k\}}{\bf 1}_{\{T_{i_0}>k'\}}f(X_k)f(X_{k'})\right] -U_{f^2}(i,i_0). 
\end{align} 
Now use the Markov property \eqref{eq:markov_property} to write 
\begin{align*}  
\mathbb{E}_{i}\left[\sum_{k=0}^{\infty}\sum_{k'=k}^{\infty}{\bf 1}_{\{T_{i_0}>k\}}{\bf 1}_{\{T_{i_0}>k'\}}f(X_k)f(X_{k'})\right] &= \sum_{k=0}^\infty \mathbb{E}_i \left[{\bf 1}_{\{T_{i_0}>k\}}f(X_k)\sum_{j=0}^\infty \mathbb{E}_{X_k}\left[ {\bf 1}_{\{T_{i_0}>j\}}f(X_j)\right]\right]\\
& =\sum_{k=0}^\infty \mathbb{E}_i \left[{\bf 1}_{\{T_{i_0}>k\}}f(X_k)U_f (X_k,i_0)\right]\\
&= U_{f U_f (\cdot,i_0)}(i,i_0).
\end{align*}
Plugging this into the right-hand side of \eqref{eq:start_square} completes the proof. 
\end{proof}
We can finally conclude with the result we were after: 
\begin{prop}
\label{prop:finally} 
Suppose $\pi(g)=0$. Then for all $i$
$$U_g^2 (i,i_0) =U_{(U_g)^2(\cdot,i_0) - (PU_g)^2(\cdot,i_0)} (i,i_0).$$ 
\end{prop} 
\begin{proof}
Below we freeze $i_0$ and treat all functions as of $i$ only. Since $\pi(g)=0$, Corollary \ref{cor:stat} gives  $U_g (i_0,i_0)=0$. With this,  Corollary \ref{cor:uh_inverse} gives 
$g= U_g -PU_g$. Plugging this into Lemma \ref{lem:U2f} we obtain 
$$ U_g^2 =U_{2 (U_g-PU_g ) U_g - (U_g - P U_g)^2} = U_{(U_g)^2 - (PU_g)^2},$$
and the result follows.
\end{proof}
If $\pi(g)=0$, then \eqref{eq:sigma_I}, Corollary \ref{cor:stat}, and Proposition \ref{prop:finally} imply
$$\sigma^2_{i_0}\left(I_{T_{i_0}}(g)\right) = \frac{\pi\left((U_g)^2(\cdot ,i_0) - (PU_g)^2 (\cdot,i_0)\right)}{\pi(i_0)}.$$ We wish to simplify the expression appearing in the numerator on the right-hand side. Let $u= U_g (\cdot,i_0)+c {\bf 1}$ for some constant $c$. Then 
$$ \pi\left( u^2 - (Pu)^2\right) = \pi \left((U_g)^2(\cdot,i_0)-(PU_g)^2(\cdot,i_0) +2c (I-P) U_{g}(\cdot,i_0)\right).$$ 
Since $\pi(g)=0$, Lemma \ref{lem:U_map} implies $(I-P)U_{g}(\cdot,i_0)=g(\cdot)$, hence for any $c$ we have
$$\pi\left( u^2 - (Pu)^2\right)=\pi\left((U_g)^2(\cdot,i_0)-(PU_g)^2(\cdot,i_0)\right).$$ 
Choosing $c=-\pi(U_{g}(\cdot,i_0))$, we find that $u$ is the unique solution to \eqref{eq:unique_sol}. Therefore, 
\begin{cor}
\label{cor:variance} 
Suppose $\pi(g)=0$ and let $u$ satisfy $(I-P)u=g$ and $\pi(u)=0$. Then 
$$\displaystyle \sigma^2_{i_0}\left(I_{T_{i_0}}(g)\right)=\frac{\pi\left( {u}^2 - (Pu)^2\right)}{\pi(i_0)}.$$
\end{cor}
\subsubsection{IID structure}
\label{sec:IID}
In this section we fix some $i_0$ in the state space. In what follows, we will suppress the dependence on $i_0$, and write $T$ for $T_{i_0}$. Define 
$$\begin{cases} T^0 = 0 \\ 
T^{m+1} = \inf\{n > T^m : X_n= i_0\}\end{cases}$$ 
Note that $T^1=T$. Since by \eqref{eq:finite_time} $\max_i \mathbb{E}_i[T_{i_0}] <\infty$,  the random variables $T^1,T^2,\dots$ are finite with probability $1$. For $m=0,1,\dots$, define 
\begin{equation}
\label{eq:Zn}
Z_m = \sum_{k=T^m}^{T^{m+1}-1}f(X_k).
\end{equation}
By implementing the Markov property \eqref{eq:markov_property}, it follows that the  RVs in the  sequence $(Z_m:m\in\Z_+)$ are independent with $Z_1,Z_2,\dots$ IID. By employing the LLN \cite[Chapter 6, Theorem 6.1]{gut_book} and CLT \cite[Chapter 7, Theorem 1.1]{gut_book} for IID sequences, we obtain 
\begin{align}  
\label{eq:prem_LLN}
& \frac{I_{T^m}(f)}{m} \to \mathbb{E}_{i_0} \left[I_{T_{i_0}}(f)\right]\mbox{, a.s.}\\
\label{eq:prem_CLT}
& \frac{I_{T^m}(f) - m\mathbb{E}_{i_0}\left[I_{T_{i_0}}(f)\right]}{\sqrt{m}} \Rightarrow N\left(0,\sigma^2_{i_0}\left(I_{T_{i_0}}(f)\right)\right), 
\end{align}
where $\sigma^2_{i_0}(I_{T_{i_0}}(f))$ is given by \eqref{eq:sigma_I}.
Now, by its very definition and Corollary \ref{cor:stat}, we have 
\begin{equation}\label{eq:pi_f}
\mathbb{E}_{i_0}\left[I_{T_{i_0}}(f)\right]=U_f(i_0,i_0)=\pi(f){\mathbb{E}_{i_0} [T_{i_0}]}.
\end{equation}
Choosing  $f\equiv {\bf 1}$ in \eqref{eq:prem_LLN} and \eqref{eq:pi_f} gives 
$$\frac{T^m}{m} = \frac{I_{T^m} ({\bf 1})}{m} \to  \mathbb{E}_{i_0} \left[T_{i_0}\right], a.s.$$
or $  T^m\sim  m \mathbb{E}_{i_0} [T_{i_0}]$ a.s.  Plugging this into \eqref{eq:prem_LLN} and \eqref{eq:prem_CLT}
gives 
\begin{align*}  
&\frac{I_{T^m}(f)}{T^m} = \frac{I_{T^m}(f)}{m} \times \frac{m}{T^m} \to \pi(f)\mbox{ a.s.}\\
& \frac{I_{T^m}(f) - m\mathbb{E}_{i_0}\left[I_{T_{i_0}}(f)\right]}{\sqrt{T^m}} = \frac{I_{T^m}(f) - m\mathbb{E}_{i_0}\left[I_{T_{i_0}}(f)\right]}{\sqrt{m}} \times \sqrt\frac{m}{T^m}  \Rightarrow  N\left(0,\frac{\sigma^2_{i_0}\left(I_{T_{i_0}}(f)\right)}{\mathbb{E}_{i_0}\left[T_{i_0}\right]}\right).
\end{align*}
Changing from $f$ to $\bar f = f - \pi(f)$, and recalling that  $0=\pi(\bar f)=\pi(i_0)\mathbb{E}_{i_0}[I_{T_{i_0}}(\bar f)]$ and $\mathbb{E}_{i_0}[T_{i_0}]= \frac{1}{\pi(i_0)}$, we proved
\begin{align}
\label{eq:prem_LLN2}
&\frac{I_{T^m}(\bar f)}{T^m} \to 0,\mbox{ a.s.}\\
\label{eq:prem_CLT2}
&\frac{I_{T^m}(\bar f)}{\sqrt{T^m}} \Rightarrow N\left(0,\pi(i_0)\sigma^2_{i_0}\left(I_{T_{i_0}}(\bar f)\right)\right).
\end{align}
In addition,  \begin{equation}
\label{eq:asymptotic_variance} 
\mathbb{E}_{i_0}\left[I_{T^m}(\bar f)^2 \right]=\mathbb{E}_{i_0}\left[T^m\right] \pi(i_0) \sigma^2_{i_0}\left(I_{T_{i_0}}(\bar f)\right).
\end{equation} 
By Corollary \ref{cor:variance}, 
\begin{equation} 
\label{eq:final_variance} 
\pi(i_0)\sigma^2_{i_0}\left(I_{T_{i_0}}(\bar f)\right)=\pi\left(({\bar U}_{\bar f})^2- (P{\bar U}_{\bar f})^2\right).
\end{equation} 
where $\bar U_{\bar f}$ is the unique solution to \eqref{eq:unique_sol} with $g= \bar f$. This expression is equal to $\sigma^2_f$ in the statement of Theorem \ref{th:CLT}. 
\subsubsection{Final step: passing to subsequences} 
\label{sec:subseq}
For each $n\in\N$, define 
$$\delta_n =\inf \{j\geq 1:X_{n+j}=i_0\}$$
and
$$V_n=\big|\{1\leq j\leq n:X_j=i_0\}\big|.$$
That is, $\delta_n$ is the time between $n$ and the first hitting of $i_0$ after time $n$, while $V_n$ is the total number of visits to $i_0$ up to and including time $n$ but not including time $0$. It is clear that for each $n$ we have $n+ \delta_n=T^{V_n+1}$.  Define $\Delta_n  = I_{n}(\bar f) - I_{T^{V_n+1}}(\bar f)$. Then 
$$|\Delta_n| \le \delta_n \|\bar f\|_\infty.$$
Next we can write
\begin{align}\label{eq:delta_split}
\frac{I_n (\bar f)}{\sqrt{n}}&= \frac{I_{T^{V_n+1}}(\bar f)+\Delta_n}{\sqrt{n}}\nonumber\\
&=\frac{I_{T^{V_n+1}}(\bar f)}{\sqrt{T^{V_n+1}}}\times \underbrace{\sqrt\frac{n+ \delta_n}{n}}_{\displaystyle\delta_{1,n}} + \underbrace{\frac{\Delta_n}{\sqrt{n}}}_{\displaystyle\delta_{2,n}}.
\end{align}
Observe that for any $i$ and $\epsilon>0$, Chebyshev's inequality gives
$$\mathbb{P}_i\left(|\Delta_n| > \epsilon \sqrt{n}\right) \le \mathbb{P}_i\left( \|\bar f\|_\infty \delta_n > \epsilon \sqrt{n}\right) \le \|\bar f\|_\infty \frac{\max_i \mathbb{E}_{i} [T_{i_0}]}{\epsilon \sqrt{n}}$$
so it follows that $\Delta_n /\sqrt{n} \to 0$ in probability. Similarly, $\delta_n/n \to 0$ in probability as well. Therefore $\delta_{2,n}\to 0$ in probability and $\delta_{1,n}\to 1$ in probability. Clearly,  \eqref{eq:finite_time} implies $\mathbb{P}_i(T_{i_0}<\infty)=1$ for any $i$. Therefore ${\bf X}$ will hit $i_0$ infinitely many times and it follows that $V_n\to\infty$ with probability $1$. Hence we can use \eqref{eq:prem_CLT2} and \eqref{eq:final_variance} in conjunction with \eqref{eq:delta_split} to finally obtain 
\begin{equation} 
\frac{I_n (\bar f)}{\sqrt{n}}\Rightarrow N(0,\sigma^2_f).
\end{equation} 

The proof that $I_n({\bar f})/n \to 0$, a.s. is similar yet simpler, and a similar argument allows us to pass from \eqref{eq:asymptotic_variance} to 
\begin{equation}
\label{eq:asymptotic_variance_final} 
\mathbb{E}_i[I_{n}(\bar f)^2 ]\sim n  \sigma^2_f.
\end{equation} 

\section{Application: Reversible case}
We say that $P$ is {\it reversible} if there exists a row vector $\mu$, not identically zero, such that the {\it detailed balance condition} holds: 
\begin{equation} 
\label{eq:DBC} \mu(i)P(i,j) = \mu(j)P(j,i).
\end{equation} 
By summing over $j$, this gives $\mu=\mu P$. As shown in Section \ref{seC:potential}, this implies $\mu$ is a constant multiple of $\pi$. When  $P$ is reversible,  the bilinear form on column vectors  $(\cdot,\cdot)_{\pi}$ given by 
\begin{equation} 
\label{eq:inner_prod} 
(u,v)_{\pi} =\sum_{i}\pi(i) u(i) v(i)
\end{equation} 
is an inner product. By construction, $P$ is symmetric with respect to this inner product. Let $\|\cdot\|_\pi= \sqrt{(u,u)_{\pi}}$. Then $\|\cdot\|_{\pi}$ is a Euclidean norm. It follows that there exists an orthonormal basis of eigenvectors for $P$. That is a sequence $\phi_1,\dots,\phi_N$ of functions (column vectors) with the properties :
\begin{enumerate} 
\item Every function from $\{1,\dots,N\}$ to $\R$ is a linear combination of $\phi_1,\dots,\phi_N$ (spanning),
\item $(\phi_s,\phi_{s'})_\pi=0$ when $s\ne s'$ (orthogonal),
\item $\|\phi_s\|_\pi=1$ for all $s$ (normalized). 
\end{enumerate} 
Denote the eigenvalue for $\phi_s$ by $\lambda_s$.
Irreducibility implies that $1$ is an eigenvalue for $P$ with eigenspace spanned by ${\bf 1}$, we will choose $\phi_1={\bf 1}$, and $\lambda_1=1$. 
Suppose now that $f = \sum_{s} c_s \phi_s $. Then $\pi(f) = c_1$, and therefore $\bar f = f-\pi(f)=\sum_{s\ne 1} c_s \phi_s$. Next let $u=\sum_{s\neq 1} \alpha_s \phi_s$ so $\pi(u)=0$. Now  $(I-P) u= \sum_{s\ne 1} \alpha_s (1-\lambda_s) \phi_s$. Therefore the unique solution to \eqref{eq:unique_sol} is given by $u$ with $\alpha_s = \frac{c_s}{1-\lambda_s}$. In addition, $Pu = \sum_{s\ne 1} \lambda_s \alpha_s \phi_s$. Due to the fact that the eigenfunctions are orthonormal, it follows that 
\begin{equation} 
\label{eq:twonorms} 
\begin{split}
 \|u\|_{\pi}^2 = \sum_{s=2}^N \frac{c_s^2}{(1-\lambda_s)^2}\\
\|P u\|_{\pi}^2  = \sum_{s=2}^N \frac{\lambda_s^2 c_s^2}{(1-\lambda_s)^2}.   
\end{split}
\end{equation}
Concluding, we have
\begin{cor}
\label{cor:reversible}
Suppose that $P$ is irreducible and reversible, that is \eqref{eq:irreducible} and \eqref{eq:DBC} hold. Let $\pi$ be its stationary distribution and let $\phi_1={\bf 1},\dots,\phi_N$ be an orthonormal basis for $\R^N$ equipped with the inner product \eqref{eq:inner_prod}, consisting of eigenfunctions for $P$. For $s=1,\dots,N$, let $\lambda_s$ be the eigenvalue for $\phi_s$. 

If  $f = \sum_{s=1}^N c_s \phi_s$, then the variance $\sigma^2_f$ in Theorem \ref{th:CLT} is equal to 
$$\sigma^2_f = \sum_{s=2}^N\frac{1+\lambda_s}{1-\lambda_s}  c_s^2.$$ 
\end{cor}
We comment that this is a very special and limited case of the well-known result of \cite{kipnis}.

Next we turn to a specific example. Suppose that $N=2$. Irreducibility is equivalent to $P(1,2)P(2,1)\ne 0$, which is also equivalent to being reversible. As a result, 
$$ \pi(1) = \frac{P(2,1)}{P(1,2)+P(2,1)}, \pi(2) = \frac{P(1,2)}{P(1,2)+P(2,1)}.$$ 
In addition, 
$$1+\lambda_2 =P(2,2)+P(1,1) = 1-P(1,2)+1-P(2,1)$$ 
therefore 
\begin{equation} 
\label{eq:lambda1} \lambda_2 = 1-\big(P(1,2)+P(2,1)\big),
\end{equation} 
Also, 
\begin{equation} 
\label{eq:theratio} \frac{1+\lambda_2}{1-\lambda_2}= \frac{2-P(1,2)-P(2,1)}{P(1,2)+P(2,1)}.
\end{equation} 
Now  $c_2^2 = \|\bar f\|_{\pi}^2$. We also have 
 $$\bar f (1) = f(1) - \frac{f(1)P(2,1)+f(2)P(1,2)}{P(1,2)+P(2,1)}= \big(f(1)-f(2)\big) \frac{P(1,2)}{P(1,2)+P(2,1)},$$ 
 and 
$$ \bar f (2) = -\big(f(1)- f(2)\big) \frac{P(2,1)}{P(1,2)+P(2,1)}.$$
Hence 
\begin{equation}\label{eq:normsquared}
\begin{split}
c_2^2= \|\bar f\|_{\pi}^2& = (f(1)-f(2))^2\frac{ P(1,2)^2 P(2,1)+P(2,1)^2 P(1,2) }{\big(P(1,2)+P(2,1)\big)^3}\\
&=(f(1)-f(2))^2 \frac{P(1,2)P(2,1)}{\big(P(1,2)+P(2,1)\big)^2}.
\end{split}
\end{equation} 
Therefore from \eqref{eq:theratio}  and \eqref{eq:normsquared} we conclude
\begin{cor}\label{cor:2_state}
Suppose that 
$$P=\left ( \begin{array}{cc} P(1,1) & P(1,2) \\ P(2,1) & P(2,2) \end{array}\right)$$ with $P(2,1)P(1,2)>0$. Then 
$$\sigma^2_f = (f(1)-f(2))^2 P(1,2)P(2,1)\frac{P(1,1)+P(2,2)}{\big(P(1,2)+P(2,1)\big)^3}.$$
\end{cor}

Specifying further, we can apply Corollary \ref{cor:2_state} to simple symmetric random walk on $\mathbb{Z}$ where 
$$P=\left ( \begin{array}{cc}1/2 & 1/2 \\ 1/2 & 1/2 \end{array}\right)$$
with $f(1)=-1$ and $f(2)=1$ to see that $\sigma^2_f=1$ as expected.

\section{Application: Finite-Memory Elephant Random Walk}
\label{sec:elephant} 
\subsection{Introduction}
The Elephant Random Walk (ERW) was introduced in \cite{Schutz04} as a model for ``non-markovian'' random walk due to ``long memory''. We will present two finite-memory versions of the Elephant Random Walk and prove the CLT for each one. We begin by introducing the original version of the ERW. 
Let $p\in [0,1]$, and suppose that we have a jar (with infinite capacity) of $\pm 1$'s starting with $N_0^+$ $+1$'s and $N_0^-$ $-1$'s in it.  We will assume $N_0^+ + N_0^-\in\N=\{1,2,\dots\}$. Set $X_0=N_0^+ - N_0^-$.

Continue inductively. Conditioned on $(N_0^+,\dots,N_n^+,N_0^-,\dots,N_n^-)$
\begin{enumerate} 
\item Sample $\pm 1$ from the jar independently from the past and replace it. 
\item Independently, with probability $p$  accept the sample.  Otherwise, ``flip it'', that is multiply it by $-1$.  The resulting sign is the  {\it new sign}, $s_{n+1}$. 
\item Set $X_{n+1} = X_n+ s_{n+1}$.
\item Add the new sign to the jar. 
\item Set $N_{n+1}^+,N_{n+1}^-$ the current number of $+1$'s and $-1$'s in the jar, respectively. 
\end{enumerate} 
The process ${\bf X}=(X_n:n\in\Z_+)$ is the Elephant Random Walk. The sign selection in item 1. features   ``long memory'', as the sign sampled can be viewed as sampling a past step, going all the back to the very beginning. When  $p=1/2$, the process is the simple symmetric RW on $\Z$ and for $p=1$, $((P_n,N_n):n\in\Z_+)$ is Polya's urn process.  The ERW exhibits an interesting set of features. When $p\ne \frac 12$ it is neither a finite-state Markov chain nor an additive functional of one. The dependence of the process on $p$ is nontrivial and even surprising with a phase transition at $p=\frac 34$, see \cite{gava}\cite{bercu}. 

A recent manuscript \cite{Gut18} studied versions of the model under an assumption of ``restricted memory'', sampling signs only from the very distant past (the first $k$ signs), or from the most recent past, where the cases of the last one and the last two signs were explicitly solved. Though convergence to a normal distribution is guaranteed, the calculation of the variance is another story. Quoting from the fifth paragraph of \cite{Gut18}:

\begin{displayquote}
\tt{The first task ... is to consider the cases when the walker only remembers the
first (two) step(s) or only the most recent (previous) step. In particular the latter case involves rather cumbersome computations and we therefore invite the reader(s) to try to push our results further.}
\end{displayquote}

In Section \ref{sec:ordered} we complete the treatment and provide a formula for the variance in the case of remembering the most recent $L$ steps. Our approach is based on Theorem \ref{th:CLT} and computation of the variance in part 2 of Lemma \ref{lem:var_lemma} through some algebraic manipulation and using exchangeability of the stationary distribution for the underlying Markov chain. This differs from the approach in \cite{Gut18}, which was based on moment calculations. 

In Section \ref{sec:disordered} we present another version of the ERW where samples are also from the last $L$ steps, and study it through Theorem \ref{th:CLT} with the variance obtained through part 1 of Lemma \ref{lem:var_lemma}. 

Let us explain the two versions we will be working on. As said, in both cases the number of ``past'' signs from which we sample is equal to some constant $L\in\N$.  To accomplish this, some signs need to be discarded and this where the two models differ. In Section  \ref{sec:disordered} we look at the case where the sign sampled is discarded and  replaced by the new sign. We call this model the ``finite disordered memory'' because no ordering of the history of the sampling is maintained. In Section \ref{sec:ordered} we keep the last $L$ ``new signs'' and therefore call it  ``finite ordered memory''. This latter version is more in-line with the ERW and extends the scope of \cite[Section 8,9]{Gut18} which studied the cases $L=1,2$. 

Like the original ERW, both of our finite memory versions reduce to simple symmetric RW on $\mathbb{Z}$ when $p=\frac{1}{2}$, hence their variances are both $1$ in this case. However, the expressions for the variances of the two models are fundamentally different:  In the disordered case, the variance is independent of $L$. In the ordered case, the variance depends on $L$ and exhibits a non-trivial limit as $L \to\infty$ which depends on $p$. 

\subsubsection{Finite disordered memory.}\label{sec:disordered}
\begin{figure}[h!]
\centering
\includegraphics[scale=0.6]{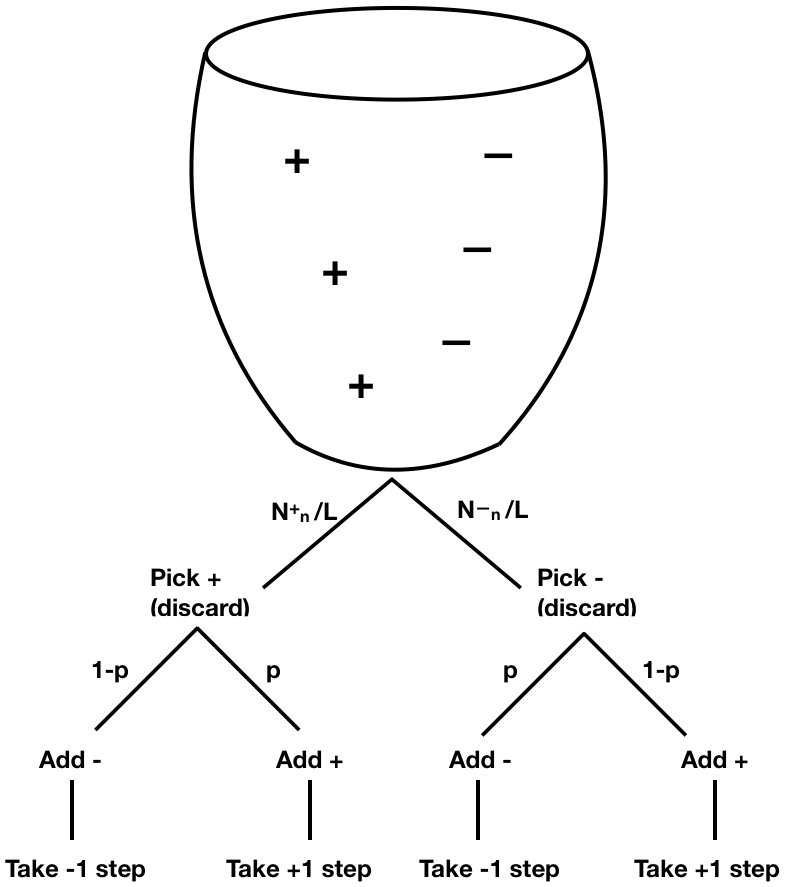}
\caption{ERW with finite disordered memory.}
\label{Urn}
\end{figure}
Start from $N_0^+ + N_0^-=L$, and $X_0$ arbitrary.  Assuming $(N_j^+,N_j^-,X_j),~j=0,\dots,n$ are defined, continue as follows: 
\begin{enumerate} 
\item Sample $\pm1$ from the jar, without replacement. 
\item With probability $p$ accept it. Otherwise, multiply by $-1$. The resulting sign in either case is {\it the new sign}, $s_{n+1}$.
\item Let $X_{n+1}=X_n+s_{n+1}$. 
\item Add the new sign to the jar. 
\end{enumerate} 
We call the resulting random walk ${\bf X}$ the ERW with finite disordered memory of length $L$ and acceptance probability $p$. 

The evolution of the signs in the jar is a Markov chain. However, it is easy to see that ${\bf X}$ is {\it not} an additive functional of the chain as it may increase by one or decrease by one while the chain will keep the same state. In order to remedy this we introduce an augmented process that will also be a Markov chain. The new process ${\bf Y}$ will  be on the state space $\Omega = \{(-1,j):j=0,\dots,L-1\}\cup\{(1,j):j=1,\dots,L\}$ and will have the following transition function 
 \begin{equation} 
 \label{eq:disordered} 
 \begin{split}  &P\big((i,j),(1,j)\big) = \frac{j}{L}p\\
 & P\big((i,j),(-1,j)\big) = \frac{L-j}{L}p. \\
  &P\big((i,j),(1,j+1)\big) = \frac{L-j}{L}(1-p)\\
  &P\big((i,j),(-1,j-1)\big) = \frac{j}{L}(1-p)
\end{split} 
\end{equation}
Observe that when $0\leq p<1$, the process $((s_n,N_n^+):n\in\Z_+)$ is an irreducible Markov chain with transition function $P$. Then ${\bf X}$ can be viewed as an additive functional of ${\bf Y}$, through the choice of the function $f(i,j) = i$.
\begin{thm}
\label{th:disordered}
The ERW with finite disordered memory of length $L$ and acceptance probability $0\leq p<1$ satisfies 
\begin{align*} \lim_{t\to\infty} \frac{X_t}{t}=0,\mbox{ a.s.}\\
\frac{X_t}{\sqrt{t}}\Rightarrow N\left(0,\frac{p}{1-p}\right).
\end{align*} 
\end{thm}
\begin{rem}\label{rem:bounded_displacement}
When $p=0$, the result is immediate since in this case the disordered memory ERW has displacement bounded by $L$. This should be contrasted with what happens in the ordered memory ERW, see Remark \ref{rem:ordered_p0}.
\end{rem}
\begin{rem}
When $p=1$, the initial proportion of signs never changes so the disordered memory ERW is simply biased random walk.
\end{rem}

Before proving Theorem \ref{th:disordered}, we begin by deriving the stationary distribution for $p$. First consider the second marginal, the number of pluses, which is a Markov chain in its own right. The stationary distribution for this chain is binomial with parameters $L$ and $\frac 12$. To see why, think of the jar as being a list of length $L$ consisting of pluses and minuses. The evolution can be described as follows: pick a location uniformly among the $L$ locations, and with probability $1-p$ flip the sign there. Clearly, if the initial distribution of the signs was IID $\mbox{Ber}(\frac 12)$, it will remain so after one step. Therefore the stationary distribution for the number of pluses is $\mbox{Bin}(L,\frac 12)$. Now to find $\pi$, let's argue heuristically through the ergodic theorem: the proportion of times the system is in state $(i,j)$ is equal to the proportion of times we reach $j$ pluses through a previous step being $i$. When $i=1$, this is equal to the sum of the proportion of times the
\begin{itemize} 
\item  number of pluses is $j$, multiplied by $\frac{j}{L}p$, representing selecting a plus and keeping it. 
\item number of pluses is $j-1$, multiplied by $\frac{L-j+1}{L}(1-p)$, representing selecting a minus and flipping it. 
\end{itemize} 
Therefore, we expect
$$ \pi(1,j) = \frac{\binom{L}{j}}{2^L} \frac{j}{L}p + \
\frac{\binom{L}{j-1}}{2^L}\frac{L-j+1}{L}(1-p) = \frac{\binom{L}{j}}{2^L}\frac{j}{L}.$$ 
Repeating the heuristics for $i=-1$, and verifying with the validity of the formula through a direct computation we arrive at 
\begin{prop}
The stationary distribution for the transition function $p$ of \eqref{eq:disordered} for $0\leq p<1$ is given by 
\begin{equation}
\label{eq:pi_LL}
\pi(i,j) = \frac{\binom{L}{j}}{L2^L}\times\begin{cases} j & i=1 \\ L-j & i=-1 \end{cases} 
\end{equation} 
\end{prop} 
\begin{proof}[Proof of Theorem \ref{th:disordered}]
Recall that  ${\bf X}$ is an additive functional for ${\bf Y}$ with the function $f(i,j)=i$. 
Next, we observe that by symmetry, $\pi(f)=0$, hence $f=\bar f$, and the first statement follows from Theorem \ref{th:CLT}. We turn now to the variance calculation. 

It is not hard to verify that the function\footnote{originally obtained using Mathematica}
\begin{equation}\label{eq:Uf_LL} {\bar U}_f(i,j) = \frac{1-2p}{1-p} \left(\frac{L}{2}-j\right) + i\end{equation} 
is the solution to \eqref{eq:unique_sol}. In addition, since 
$$ \left(P{\bar U}_f\right)(i,j) =  {\bar U}_f(i,j) - i,$$ 
we have 
$$ \left({\bar U}_f\right)^2 - \left(P{\bar U}_f\right)^2 = 2{\bar U}_f (i,j) i -1.$$ 
Now since $\pi\left({\bar U}_f\right)=0$, it follows that 
$$\pi\left({\bar U}_f \times i \right) = \pi\left({\bar U}_f \times (i+1)\right)=2 \sum_{(1,j)}\pi(1,j)U_f(1,j).$$
Hence, 
$$\sigma^2_f =4 \sum_{(1,j)} \pi(1,j) {\bar U}_f (1,j) -1.$$ 
Thus, by \eqref{eq:pi_LL} and by \eqref{eq:Uf_LL} we obtain 
$$\sigma^2_f = 4 \sum_{j=1}^L \frac{\binom{L}{j}}{2^L} \frac{j}{L}\left(\frac{1-2p}{1-p}\left(\frac{L}{2}-j\right)+1\right)-1.$$ 
Now let $B\sim \mbox{Bin}(L,\frac 12)$. Then 
\begin{align*} \sigma^2_f&= 4 E\left[ \frac{B}{L} \left(\frac{1-2p}{1-p} \left(\frac{L}{2}-B\right)+1\right)\right]-1\\
 & =4 \left ( \frac{1-2p}{1-p} \left(\frac L4  - \frac 14 (1+L)\right)+\frac 12\right) -1\\
 & =2-  \frac{1-2p}{1-p}-1\\
 & =\frac{p}{1-p}. 
\end{align*}
\end{proof} 

\subsubsection{Finite ordered memory.} 
\label{sec:ordered}
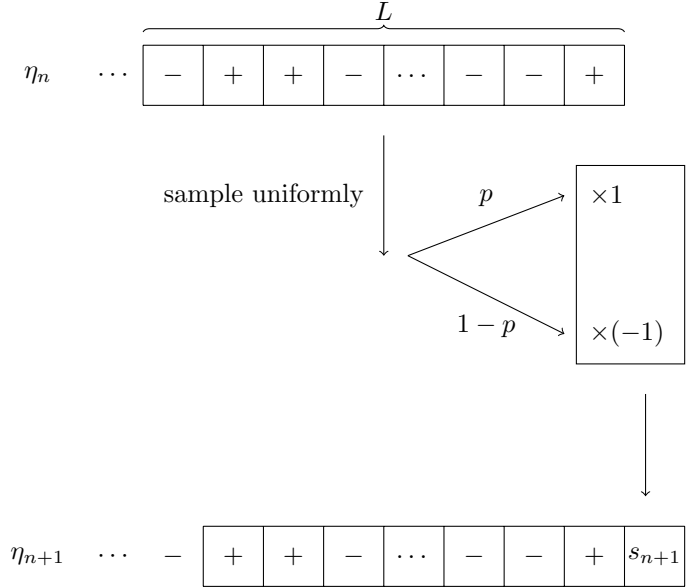
\begin{figure}[h!]
\centering
\begin{tikzpicture}[scale=.8]
     \draw[decoration={brace,raise=5pt},decorate] (1,10) -- node[above=6pt] {$L$} (9,10);
    \draw (1, 10) grid (9, 9);
    \setrow{\cdots}{-}{+} {+}{-}{\dots } {-}{-}{+}
    \node[anchor=center] at (-0.75, 9.5) {$\eta_n$};
    \draw[->] (5,8.5) -- node[left=4pt] {sample uniformly} (5,6.5);
    \draw[->] (5.4,6.5) -- node[above=4pt] {$p$}  (8,7.5) node[right=6pt] {$\times 1$};
     \draw[->] (5.4,6.5) -- node[below=4pt] {$1-p$}  (8,5.2) node[right=6pt] {$\times (-1)$};
       \draw (8.2,8)  rectangle (10,4.7);
      \draw[->] (9.35,4.2) -- (9.35,2.5);
      \draw (2, 2) grid (10, 1);
      \setcounter{row}{8}
      \setrow{\cdots}{-}{+} {+}{-}{\dots } {-}{-}{+}
      \node[anchor=center] at (9.5, 1.5) {$s_{n+1}$};
      \node[anchor=center] at (-0.75, 1.5) {$\eta_{n+1}$};
  \end{tikzpicture}
\caption{One-step transitions for the process $\eta_{\cdot}$}
\label{Urn2}
\end{figure}
Here the jar will be an (ordered) list of length $L$ consisting of $\pm1$'s.
Let $\eta_n=(\eta_n(0),\eta_n(1),\dots,\eta_n(L-1))$ be the list at time $n$.  Start with an arbitrary list $\eta_0$, and arbitrary $X_0$. Assuming $(\eta_j,X_j),~j=0,\dots,n$ are defined, continue as follows: 
\begin{enumerate} 
\item Sample $\pm1$ with probability proportional to the number of the items of the same type in the list. 
\item Accept the sampled sign with probability $p$. Otherwise, multiply by $-1$. The resulting sign is the new sign, $s_{n+1}$. 
\item Let $X_{n+1}=X_n+s_{n+1}$. 
\item Let $\eta_{n+1} = (\eta_n(1),\dots,\eta_n(L-1),s_{n+1})$. 
\end{enumerate}
Here we always discard the ``oldest'' sign in the list. The evolution of the list is a Markov chain, and ${\bf X}$ is an additive functional of the chain. We call it the ERW with finite ordered memory of length $L$ and acceptance probability $p$. 
\begin{thm}
\label{th:ordered} 
The ERW with finite ordered memory of length $L$ and acceptance probability $0\leq p<1$ satisfies
\begin{align*}
 \lim_{t\to\infty} \frac{X_t}{t}=0,\mbox{ a.s.}\\
 \frac{X_t}{\sqrt{t}}\Rightarrow N(0, \sigma^2),
\end{align*} 
where 
\[ \sigma^2 = 
\frac{L-1+2p}{2(1-p)\Big(2(1-p) L +2p-1\Big)}.
\]
\end{thm}

\begin{rem}\label{rem:equivalence}
When $L=1$, the ordered memory and disordered memory ERW are identical so both variances agree in this case.
\end{rem}
\begin{rem}
In the limit as $L\to\infty$, the variance approaches $\frac{1}{4(1-p)^2}$.
\end{rem}
\begin{rem}\label{rem:ordered_p0}
When $p=0$, the variance is $\frac{L-1}{4L-2}$. This should be contrasted with what happens in the disordered memory ERW, see Remark \ref{rem:bounded_displacement}.
\end{rem}
\begin{rem}
When $p=1$, the states where the list is all $+1$'s or all $-1$'s are both absorbing states. Consequently, the ordered memory ERW will eventually become degenerate and only take steps in one direction. The distribution of the limiting direction will depend on the initial state of the list.
\end{rem}
To study this model, we introduce a Markov chain on the state space $\Omega= \{0,1\}^L$, where for $\omega\in \Omega$, we write $\omega= (\omega_0,\dots,\omega_{L-1})$, and $|\omega| = \sum_j \omega (j)$. The corresponding Markov chain will be denoted by  $(\eta_n:n\in\Z_+)$, where for each $n$, $\eta_n$ is an element in $\Omega$. The transition function $Q$ is given by the following formula: 
\begin{equation}
    \label{eq:ordered} 
Q \Big(\eta,\big(\eta(1),\dots,\eta(L-1),s \big) \Big) =\begin{cases} 
\frac{1}{L}\big (  |\eta| p +(L-|\eta|)(1-p)\big) & s=1 \\
\frac{1}{L} \big((L-|\eta|)  p +|\eta| (1-p)\big) & s=0\\
0 & \mbox{otherwise}\end{cases}
\end{equation} 
The connection to the ERW with finite ordered memory is clear: $2\eta_{n+1}(L-1)-1$ represents the step $X_{n+1}-X_n$. The reason why we choose $0$ and $1$ rather than $\pm1$ is to simplify notation. It will also be more convenient to work with the left-most entries of $\eta_n$, hence we note here that 
\begin{equation}\label{eq:eta_to_X}
X_{n+1}-X_n=2\eta_{n+L}(0)-1.
\end{equation}

We now find the stationary distribution for $Q$. We do this by making an assumption and verifying that the resulting distribution satisfies the requirements. We will assume that the stationary distribution $\pi$ is only a function of the number of $1$'s (or equivalently, the number of $0$'s). That is $\pi (\omega) = g(|\omega|)$ for some $g$. For this reason, we will sometimes abuse notation for $k\in \{0,\dots,L\}$ and write $\pi(k)$ meaning $\pi (\omega)$ for $\omega$ with $|\omega|=k$. With this assumption, we are led to the following expression for $\pi$:
\begin{prop}
Suppose $0\leq p<1$. Then the  stationary distribution $\pi$  for $Q$ defined in \eqref{eq:ordered} is
\begin{equation}\label{eq:stat}
    \pi (\omega) = Z^{-1} \prod_{j=0}^{|\omega|-1} \frac{c_j}{c_{L-j-1}},
\end{equation}
where $Z$ is a normalization constant and
\begin{equation}\label{eq:c_j}
c_j=(1-p)\left(1-\frac{j}{L}\right)+\frac{j}{L}p.
\end{equation}
\end{prop}
\begin{proof}
For any $\omega$, transition into $\omega$ can occur from only two states: $\omega_1= (1,\omega(0),\dots,\omega(L-2))$ or $\omega_{0} = (0,\omega(0),\dots,\omega(L-2))$. Let's suppose first that $\omega(L-1)=1$ and $|\omega|=k$ for $k\in\{1,\dots,L\}$. Then the equation that $\pi$ needs to satisfy is 
$$ \pi(\omega_1) Q (\omega_1,\omega) + \pi (\omega_{0}) Q(\omega_{0},\omega) =\pi (\omega).$$ 
Since $|\omega_1|= |\omega|=k$, we need to check that  
$$ \prod_{j=0}^{k-1} \frac{c_j}{c_{L-j-1}} Q (\omega_1,\omega) + \prod_{j=0}^{k-2} \frac{c_j}{c_{L-j-1}} Q (\omega_{0},\omega) = \prod_{j=0}^{k-1} \frac{c_j}{c_{L-j-1}},$$ 
or equivalently
\begin{equation}\label{eq:stat_check}
Q (\omega_{0},\omega)c_{L-k} = \big(1-Q (\omega_1,\omega)\big)c_{k-1}.
\end{equation}

In order to prove that \eqref{eq:stat_check} holds, we start with some observations about the $c_j$'s. While the form of \eqref{eq:c_j} was chosen to make obvious the fact that $c_j>0$ for $0\leq j\leq L-1$ as long as $0\leq p<1$, the following formulation will be more convenient in what follows:
\begin{equation}\label{eq:c_j_alt}
c_j=1-p+\frac{j}{L}(2p-1).    
\end{equation}
From \eqref{eq:c_j_alt} we can observe the symmetry property
\begin{align}\label{eq:symmetry}
1-c_j &= p - \frac{j}{L}(2p-1)\nonumber \\
&= p - (2p-1) + \frac{L-j}{L}(2p-1)\nonumber \\
&=c_{L-j}.
\end{align}

Using \eqref{eq:ordered} along with \eqref{eq:c_j_alt} we can now write 
\begin{align*}
Q(\omega_1,\omega) &= \frac{1}{L} \big ( k p +(L-k)  (1-p)\big)\\
&=1-p+\frac{k}{L}(2p-1)\\
&=c_k
\end{align*}
and 
\begin{align*}
Q(\omega_{0},\omega) &= \frac{1}{L}\big ( (k-1) p + (L-k+1)(1-p)\big)\\
&=1-p+\frac{k-1}{L}(2p-1)\\
&=c_{k-1}.
\end{align*}
From this it is straightforward to verify \eqref{eq:stat_check} using the symmetry property \eqref{eq:symmetry}.

Next we turn to the case where $\omega(L-1)=0$, so now $k\in \{0,\dots, L-1\}$ and $|\omega_0|=|\omega|=k$. Similarly to the previous case, we need to check that  
$$ \prod_{j=0}^{k} \frac{c_j}{c_{L-j-1}} Q (\omega_1,\omega) + \prod_{j=0}^{k-1} \frac{c_j}{c_{L-j-1}} Q (\omega_{0},\omega) = \prod_{j=0}^{k-1} \frac{c_j}{c_{L-j-1}},$$ 
or equivalently
\begin{equation}\label{eq:stat_check_alt}
Q (\omega_{1},\omega)c_{k} = \big(1-Q (\omega_0,\omega)\big)c_{L-k-1}.
\end{equation}
Using \eqref{eq:ordered} and \eqref{eq:c_j_alt} as before, we can write 
$$Q(\omega_1,\omega) =1-c_{k+1}$$
and 
$$Q(\omega_{0},\omega) =1-c_{k}.$$
Hence \eqref{eq:stat_check_alt} follows from the symmetry property \eqref{eq:symmetry}.
\end{proof}

Using the stationary distribution \eqref{eq:stat}, we can recover some interesting and important identities. Note that for $n\ge m$, $\mathbb{E}_{\pi} [\eta_m(0) \eta_n(0)]$ depends only on the difference $n-m$, and is equal to 
$\rho_{n-m} := \mathbb{E}_{\pi} [\eta_0(0)\eta_{n-m}(0)]$. Clearly, $\rho_0=\frac 12$. In addition, since $\pi$ is only a function of the number of $1$'s, it follows that $\rho_{1}=\dots=\rho_{L-1}$. We turn to the general case. Suppose $n \ge L$. Then 
$$\rho_n =  \mathbb{E}_{\pi}[\eta_0(0) \eta_n(0)] =\mathbb{E}_{\pi}\left[ \eta_0(0) \mathbb{E}[\eta_n(0) | \eta_{n-1},\eta_{n-2},\dots,\eta_{n-L},\dots,\eta_0]\right].$$ 
The conditional expectation is equal to $\frac{2p-1}{L} (\eta_{n-L}(0)+\dots+ \eta_{n-1}(0))+  (1-p)$. Therefore, we found that 
$$ \rho_n =\frac{2p-1}{L} \sum_{j=1}^L \rho_{n-j}+ \frac{1-p}{2}.$$
Let $\bar\rho_n=\rho_n - \frac 14$. Then 
\begin{equation} 
\label{eq:barrho} 
\bar\rho_n = \frac{2p-1}{L} \sum_{j=1}^L \bar \rho_{n-j}.
\end{equation} 
From this it is clear that $\bar\rho_n \to 0$ exponentially fast when $0<p<1$. In fact, this also occurs when $p=0$ as long as $L\geq 2$. To see this, we examine the matrix of the linear recursion and its square. Let $A$ be the recursion matrix when $p=0$. Then we have 
$$
A=\left[\begin{array}{cccc}
-\frac{1}{L}&-\frac{1}{L}&\cdots &-\frac{1}{L}\\
1&& &\\
&\ddots &&\text{\huge0}\\
\text{\huge0}&&1 &\\
\end{array}\right]
~~\text{ and }~~
A^2=\left[\begin{array}{cccc}
\frac{1-L}{L^2}&\cdots&\frac{1-L}{L^2} &\frac{1}{L^2}\\
-\frac{1}{L}&-\frac{1}{L}&\cdots &-\frac{1}{L}\\
1& &&\\
&\ddots&&\text{\huge0}\\
\text{\huge0}&&1&
\end{array}\right].
$$
Let $\sigma(\cdot)$ denote spectral radius and $|\cdot|$ denote the entry-wise absolute value. By viewing $|A^2|$ as the transition matrix of a (sub) Markov chain, it is clear that $|A^2|$ is an irreducible nonnegative matrix. Additionally, the maximum row sum is $1$ and the minimum row sum is strictly less than $1$. Hence the Perron-Frobenius theorem implies that $\sigma(|A^2|)<1$. Now it follows from Lemma 2.4 of \cite{Varga} that
\[
\sigma(A)=\sqrt{\sigma\left(A^2\right)}\leq\sqrt{\sigma\left(\left|A^2\right|\right)}<1.
\]
Thus $\bar\rho_n \to 0$ exponentially fast when $0\leq p<1$ and $L\geq 2$.

\begin{proof}[Proof of Theorem \ref{th:ordered}] 
In light of Remark \ref{rem:equivalence}, we can exclude the case $L=1$ so from now on we assume $L\geq 2$. The first statement follows from symmetry and we will omit the proof. For the second, we will study the additive functional $S_n := \sum_{j=0}^{n-1} \eta_j (0)$. From \eqref{eq:eta_to_X}, we see that 
\begin{align*}
X_n&=X_0+2\sum_{j=0}^{n-1}\eta_{j+L}(0)-n\\
&=X_0+2\left(S_n-\frac{n}{2}\right)+\underbrace{2\left(\sum_{j=n}^{n-1+L}\eta_j(0)-\sum_{j=0}^{L-1}\eta_j(0)\right)}_{\displaystyle\Delta_n}
\end{align*}
with $|\Delta_n|\leq 4L$. As a result, the variance we obtain for ${\bf S}=(S_n:n\in\Z_+)$ as an additive functional for $\eta$ will have to be multiplied by $4$. 

Due to Theorem \ref{th:CLT}, we know that $(S_n-\frac{n}{2})/\sqrt{n}$ converges weakly to $N(0,\sigma^2)$. From the discussion in Sections \ref{sec:IID} and \ref{sec:subseq}, it is easy to conclude that the expression for the variance in Theorem \ref{th:CLT} is asymptotically equivalent to $\mathbb{E}_{\pi}[ (S_n- \mathbb{E}_{\pi} [S_n])^2] /n$. Hence we will find the limit of this expression as $n\to\infty$ and use it as the variance, instead of solving \eqref{eq:unique_sol}. 

 Clearly,  $\mathbb{E}_{\pi}[S_n] = \frac{n}{2}$. Next, 
\begin{align*} \mathbb{E}_{\pi}[S_n^2]&=n \rho_0 + 2 \sum_{i=0}^{n-2} \left(\sum_{j=i+1}^{n-1} \rho_{j-i}\right)\\
& = \frac{n}{2}+ 2\sum_{i=0}^{n-2}\left(\sum_{k=1}^{n-1-i}\rho_k\right)\\
& = \frac{n}{2}+ 2\sum_{k=1}^{n-1}\left(\sum_{i=0}^{n-k-1}\rho_k\right) \\
& = \frac{n}{2} +2\sum_{k=1}^n (n-k) \rho_k\\
& =\frac{n}{2} +2\sum_{k=1}^n (n-k) \bar\rho_k + \frac{1}{2} \sum_{k=1}^n (n-k) \\
& = \frac{n}{2} +2 n \sum_{k=1}^n\bar\rho_k -2\sum_{k=1}^n k \bar\rho_k+\frac{(n-1)n}{4}.
\end{align*}
Now using the exponential decay of the $\bar\rho_k$'s allows us to write
\begin{align*}
\mathbb{E}_{\pi}[S_n^2]& = \frac{n^2}{4} + \frac{n}{4} +2n \sum_{k=1}^n\bar\rho_k + O(1) \\
& = \mathbb{E}_{\pi}[S_n]^2 + n \left (\frac 14 + 2 \sum_{k=1}^n\bar \rho_k\right) + O(1). 
\end{align*} 
And so, 
\begin{equation} 
\label{eq:sn_variance} \sigma^2_{S_n} \sim n \left( \frac 14+ 2 \sum_{\ell=1}^\infty \bar \rho_{\ell} \right).
\end{equation} 
To calculate this quantity, we will use the recurrence relation \eqref{eq:barrho}. Let $a=\sum_{\ell=1}^\infty \bar\rho_\ell$. Then 
\begin{align*} a&= (L-1) \bar\rho_1+ \sum_{\ell=L}^\infty \bar\rho_\ell \\
 & =(L-1) \bar\rho_1 + \frac{2p-1}{L}\sum_{\ell= L}^{\infty}\sum_{j=1}^L \bar \rho_{\ell-j}\\
 & = (L-1) \bar \rho_1 + \frac{2p-1}{L} \sum_{\ell=L}^\infty \left(\sum_{k=\ell-L}^{\ell-1} \bar \rho_k\right)\\
 &= (L-1) \bar\rho_1 + \frac{2p-1}{L}\sum_{k=0}^\infty \left(\sum_{\ell=\max(k+1,L)}^{k+L}\bar\rho_k\right)\\
&=  (L-1) \bar\rho_1 +\frac{2p-1}{L}\left ( \sum_{k=0}^{L-1}(k+1) \bar\rho_k + L \sum_{k=L}^\infty \bar\rho_k\right)\\
& = (L-1) \bar\rho_1 + \frac{2p-1}{L}\left ( \bar\rho_0 +  \sum_{k=1}^{L-1}(k+1-L) \bar\rho_k + L \sum_{k=1}^\infty \bar\rho_k\right)\\
& = (L-1) \bar\rho_1 +\frac{2p-1}{L}\left ( \frac 14 - \sum_{j=1}^{L-2}j \bar \rho_1 + L a\right)\\
& = (L-1) \bar\rho_1  + \frac{2p-1}{4L} - (2p-1)\frac{(L-2)(L-1)}{2L}\bar\rho_1+  (2p-1) a.
 \end{align*} 
 Therefore, 
 \begin{equation}
 \label{eq:newais}
 a (2-2p) =(L-1) \left ( 1 -\frac{2p-1}{2} \frac{L-2}{L}\right) \bar\rho_1 + \frac{2p-1}{4L}.
 \end{equation} 

Next we calculate $\rho_1$. This is essentially the same computation as before, shifted: 
$$\rho_1 = \mathbb{E}_{\pi}[\eta_0(L-1)\eta_1(L-1)].$$ 
We will start by using the Markov property to write
$$\mathbb{E}[\eta_1(L-1)|\eta_0]=(\eta_0(0)+ \dots + \eta_0(L-1))\frac{2p-1}{L}+(1-p).$$ 
Therefore 
$$\rho_1 = \frac{2p-1}{L} \sum_{j=0}^{L-1} \rho_j  + \frac{1-p}{2}.$$ 
That is, 
$$ \rho_1 = \frac{2p-1}{L} \left ( \frac 12 + (L-1) \rho_1\right) + \frac{1-p}{2}.$$ 
Solving for $\rho_1$ we have 
$$ \rho_1 \left(1 - \frac{2p-1}{L}(L-1)\right) = \frac{2p-1}{2L} + \frac{1-p}{2}.$$ 
Or 
$$\rho_1 ( L - (2p-1) (L-1)) = \frac{2p-1}{2} + \frac{(1-p)L}{2},$$
or 
$$\rho_1 (2L(1-p) +2p-1)  =\frac {1}{2} \left ( 2p-1 + (1-p)L\right).$$ 
This means 
$$\bar\rho_1 (2L(1-p) + 2p-1) = \frac 12 \left ( 2p-1 + (1-p)L - L(1-p)-\frac{2p-1}{2}\right).$$ 
That is
\begin{equation} 
\label{eq:barrho1}
\bar\rho_1 = \frac{2p-1}{4(2L(1-p)  + 2p-1)}.
\end{equation}
Plugging this into \eqref{eq:newais} and \eqref{eq:sn_variance},   we obtain 
$$\sigma^2_{S_n/n} \sim \frac 14+ 2a = \frac14+  \frac{L-1}{1-p} \left ( 1 -\frac{2p-1}{2} \frac{L-2}{L}\right)\frac{2p-1}{4(2L(1-p)  + 2p-1)}  + \frac{2p-1}{4L(1-p)}.
$$
We multiply this by $4$, combine the fractions and simplify to  obtain
$$\frac{L-1+2p}{2(1-p)(2(1-p)L+2p-1)}$$
which is equivalent to the expression for $\sigma^2$ in the statement of  Theorem \ref{th:ordered}.
\end{proof} 

\bibliographystyle{alpha}
\bibliography{bibliography}

\end{document}